\documentclass[12pt,reqno]{amsart}
\usepackage{amsthm,amssymb,amsmath,microtype}
\usepackage{mathtools}
\usepackage[shortlabels]{enumitem}
\usepackage{shuffle}
\usepackage{bm}
\usepackage[utf8]{inputenc}
\usepackage[english]{babel}
\usepackage{listings}
\usepackage{geometry}
\usepackage{url}
\usepackage{graphicx}
\usepackage{setspace}
\usepackage{esint}
\usepackage[
pagebackref,
bookmarks=true,         
bookmarksnumbered=true, 
colorlinks=true,
pdfstartview=FitV,
linkcolor=blue, citecolor=blue, urlcolor=blue]{hyperref}

\def\RR{{\mathbb{R}}}

\def\EE{{\mathbb{E}}}

\def\caa{{\mathcal A}}

\def\cee{{\mathcal E}}

\newcommand{\HH}{\mathfrak H}

\def\lt{\left}
\def\rt{\right}
\theoremstyle{plain}
\newtheorem{theorem}{Theorem}

\newtheorem{lemma}[theorem]{Lemma}

\theoremstyle{definition}

\theoremstyle{remark}

\usepackage[]{enumitem}
\usepackage[
backrefs,
abbrev,
]{amsrefs}
\usepackage{prettyref}
\newrefformat{chap}{Chapter \ref{#1}}
\newrefformat{sub}{Subsection \ref{#1}}
\newrefformat{prop}{Proposition \ref{#1}}
\newrefformat{rem}{Remark \ref{#1}}
\newrefformat{cor}{Corollary \ref{#1}}
\newrefformat{def}{Definition \ref{#1}}
\newrefformat{ex}{Example \ref{#1}}
\newrefformat{lem}{Lemma \ref{#1}}
\newrefformat{eq}{Equation \eqref{#1}}
\newrefformat{con}{Condition \ref{#1}}
\begin{document}

\title[] {A remark on a result of Xia Chen}

\author[K. L\^e]{ Khoa L\^e}

\address{Mathematical Sciences Research Institute, Berkeley, California, USA}
\email{khoale@ku.edu}
\date{September 25, 2015}
\begin{abstract}
	We consider the parabolic Anderson model which is driven by a Gaussian noise fractional in time and having certain scaling property in the spatial variables. Recently, Xia Chen has obtained exact Lyapunov exponent for all moments of positive integer orders. In this note, we explain how to extend Xia Chen's result for all moments of order $p$, where $p$ is any real number at least 2. 
\end{abstract}
\maketitle
	The parabolic Anderson model takes the form
	\begin{equation}\label{SHE}
		\partial_t u_ \lambda-\frac \Delta2 u_ \lambda=\sqrt \lambda u \dot{W}\,,
	\end{equation}
	where $t\ge0$, $x\in\RR^d$, $\lambda>0$ and $\dot{W}$ is a centered generalized Gaussian field. The covariance of $\dot{W}$ is given by
	\begin{equation}
		\EE[\dot{W}(t,x)\dot{W}(s,y)]=\gamma_0(t-s)\gamma(x-y)\,.
	\end{equation}
	The parameter $\lambda$ represent the intensity of the noise.
	We assume throughout the note that the initial datum of \eqref{SHE} is the constant 1. Equation \eqref{SHE} has a unique random field solution $\{u_ \lambda(t,x);t\ge0,x\in\RR^d\}$ satisfying
	\begin{equation}
		u_ \lambda(t,x)=1+\sqrt \lambda\int_0^t\int_{\RR^d}p_{t-s}(x-y)u_ \lambda(s,y)W(ds,dy)\,.
	\end{equation}
	It is known that $u_ \lambda(t,x)$ has finite moment of all positive orders. In addition, the moments of natural order can be expressed explicitly through a Feynman-Kac-type formula (c.f. \cite{hunualartsong}), namely
	\begin{equation}\label{eqn:FKn}
		\EE u^n_ \lambda(t,x)=\EE\exp\left\{\lambda\sum_{1\le j<k\le n}\int_0^t\int_0^t \gamma_0(s-r)\gamma(B_j(s)-B_k(r))dsdr \right\}\,,
	\end{equation}
	where $B_1,\dots,B_n$ are independent $d$-dimensional Brownian motions. We assume the following scaling condition
	\begin{enumerate}[(S)]
	 	\item\label{con:scaling} There exist $\alpha_0\in(0,1)$ and $\alpha\in(0,2)$ such that $\gamma_0(t)=|t|^{-\alpha_0}$ and $\gamma(c x)=c^{-\alpha}\gamma(x) $ for all $x\in\RR^d$ and all positive numbers $c$.
	\end{enumerate} 
	It is shown by Xia Chen in \cite{ChenXia15} under condition \ref{con:scaling} and other technical assumptions that for every $x\in\RR^d$ and every integer $n\ge1$,
	\begin{equation}\label{eqn:nmm}
		\lim_{t\to\infty}t^{-\frac{4- \alpha-2 \alpha_0}{2- \alpha}}\log \EE u^n_ \lambda(t,x)=n\left(\frac{n-1}2 \right)^{\frac2{2- \alpha}}\cee (\lambda)\,.
	\end{equation}
	where $\cee(\lambda)$ is the variation
	\begin{equation}
		\cee(\lambda)=\sup_{g\in\caa_d}\lt\{\lambda\int_0^1\int_0^1\int_{\RR^d\times\RR^d}\frac{\gamma(x-y)}{|s-r|^{-\alpha_0}}g^2(s,x)g^2(r,y)dxdydrds-\frac12\int_0^1\int_{\RR^d}|\nabla_x g(s,x)|^2dxds \rt\}\,.
	\end{equation}
	In the above expression, $\caa_d$ is the class of functions defined as
	\begin{equation*}
		\caa_d=\lt\{g\colon g(s,\cdot)\in W^{1,2}(\RR^d)\mbox{ and }\int_{\RR^d}g(s,x)dx=1\,\forall 0\le s\le 1 \rt\}\,.
	\end{equation*}
	Under scaling assumption \ref{con:scaling}, we always have
	\begin{equation}\label{eqn:Escaling}
		\cee(\lambda)=\lambda^{\frac2{2- \alpha}} \cee(1)\,.
	\end{equation}
	Our main purpose in this note is to show that \eqref{eqn:nmm} holds for any real numbers $n\ge2$. 
	\begin{theorem}\label{thm:1}
		For every $x\in\RR^d$ and every real number $p\ge2$, we have
		\begin{equation}\label{eqn:pmm}
			\lim_{t\to\infty}t^{-\frac{4- \alpha-2 \alpha_0}{2- \alpha}}\log \EE u^p_ \lambda(t,x)=p\left(\frac{p-1}2 \right)^{\frac2{2- \alpha}}\cee(\lambda)\,.
		\end{equation}
	\end{theorem}
	We note that in the proof of the lower bound of \eqref{eqn:nmm} in \cite{ChenXia15}, we can replace $n$ by any real number $n>1$. Therefore, we always have
	\begin{equation}\label{eqn:lowerp1}
		\liminf_{t\to\infty}t^{-\frac{4- \alpha-2 \alpha_0}{2- \alpha}}\log \EE u^p_ \lambda(t,x)\ge p\left(\frac{p-1}2 \right)^{\frac2{2- \alpha}}\cee(\lambda)\,,\mbox{ for all } p>1\,.
	\end{equation}
	Base on this evident, we conjecture that \eqref{eqn:pmm} holds for every real number $p\ge1$ (note that the case $p=1$ is trivial).

 	On the other hand, Xia Chen's proof of the upper bound of \eqref{eqn:nmm} relies heavily on the Feynman-Kac moment representation \eqref{eqn:FKn}, which is meaningful only for moments of positive integer orders. 
 	To obtain the upper bound of \eqref{eqn:pmm} which is valid for any real number $p\ge2$, the key observation is the following inequality, proved in \cite{HLN15}.
 	\begin{lemma}
 		For every $q\ge p>1$, we have
 		\begin{equation}\label{eqn:qp}
 			\lt\|u_{\frac{p-1}{q-1}\lambda}(t,x)\rt\|_{L^q(\Omega)}\le\lt\|u_{\lambda}(t,x)\rt\|_{L^p(\Omega)}
 		\end{equation}
 		for every $t>0$ and  $x\in\RR^d$.
 	\end{lemma}
 	Since this inequality is implicit in \cite{HLN15}, we reproduce the proof here. 
 	\begin{proof}
 		Let $\{P_\tau\}_{\tau\ge0}$ denote the Ornstein-Uhlenbeck semigroup in the Gaussian space associated with the noise $W$.  For a bounded measurable function $f$ on $\RR^{\HH}$,   we have the following Mehler's formula
		\begin{equation*}
			P_\tau f(W) =\EE'f(e^{-\tau}W+\sqrt{1-e^{-2\tau}}W')\,,
		\end{equation*}
		where  $W'$  an independent copy of $W$, and $\EE'$ denotes the expectation with respect to $W'$. 
 		For each $\tau\ge0$, let $u_{\tau,\lambda}$ be the solution to  equation \eqref{SHE} driven by the space-time Gaussian field $\sqrt{\lambda} (e^{-\tau}W+\sqrt{1-e^{-2\tau}}W')$, with initial condition $u_0=1$. That is,
		\begin{align*}
			u_{\tau,\lambda}(t,x)&=1
			\\&\quad+\sqrt\lambda\int_0^t\int_{\RR ^d} p_{t-s}(x-y)u_{\tau,\lambda}(s,y) (e^{-\tau}W(ds,dy)+\sqrt{1-e^{-2\tau}}W'(ds,dy))
			\,.
		\end{align*}
		From Mehler's formula we see that $P_\tau u_\lambda=\EE_{W'}[ u_{\tau,\lambda}]$ satisfies the equation
		\begin{equation*}
			P_\tau u_\lambda(t,x)=1+\sqrt\lambda e^{-\tau}\int_0^t\int_{\RR^d} p_{t-s}(x-y)P_\tau u_\lambda(s,y)W(ds,dy)\,.
		\end{equation*}
		In other words, $P_ \tau u_ \lambda$ is another solution of \eqref{SHE} with $\lambda$ being replaced by $\lambda e^{-2 \tau}$. By uniqueness, we conclude that $P_\tau u_\lambda=u_{e^{-2\tau}\lambda}$. On the other hand, it is well-known that the Ornstein-Uhlenbeck semigroup verifies the following hypercontractivity inequality
		\begin{equation}\label{ineq.hypercontractivity}
			\|P_{\tau}f\|_{q(\tau)}\le\|f\|_{p}
		\end{equation}
		for all $1<p<\infty$ and  $\tau\ge0$, where  $q(\tau)=1+e^{2 \tau}(p-1)$. Hence, applying the hypercontractivity property to our situation yields
		\begin{equation}\label{tmp:utl}
			\|u_{e^{-2\tau}\lambda}(t,x)\|_{q(\tau)}\le\|u_{\lambda}(t,x)\|_p
		\end{equation}
		for all $t>0,x\in\RR^d$ and $\tau\ge 0$. The result follows by choosing $\tau$ so that $q(\tau)=q$, that is $e^{-2 \tau}=\frac{p-1}{q-1}$. 
 	\end{proof}
 	\begin{proof}[Proof of Theorem \ref{thm:1}]
 	 	The lower bound has already observed in \eqref{eqn:lowerp1}. For the upper bound, let us choose $p=2$ in \eqref{eqn:qp} to obtain
	 	\begin{equation*}
	 		\lt\|u_{\frac{\lambda}{q-1}}(t,x)\rt\|_{L^q(\Omega)}\le\lt\|u_{\lambda}(t,x)\rt\|_{L^2(\Omega)}
	 	\end{equation*}
	 	for all $q\ge2$. We can apply \eqref{eqn:nmm} to get for all $q\ge2$ and $\lambda>0$
	 	\begin{equation*}
	 		\limsup_{t\to\infty}t^{-\frac{4- \alpha-2 \alpha_0}{2- \alpha}}\log \EE u^q_ {\frac{\lambda}{q-1}}(t,x)\le2^{-\frac2{2- \alpha}} \cee(\lambda)\,,
	 	\end{equation*}
	 	which, by the scaling relation \eqref{eqn:Escaling}, is equivalent to the upper bound for \eqref{eqn:pmm}. 
 	\end{proof}
 	We conclude with the following remark. From the expression \eqref{eqn:pmm}, the inequality \eqref{eqn:qp} is indeed equality in the limit $t\to\infty$. The situation is quite different if $W$ is a space-time white noise. In fact, in this case, it is proved in \cite{Ch15} that for every integer $n\ge1$,
 	\begin{equation}\label{eqn:whitenmm}
 		\lim_{t\to\infty}\frac1t\log\EE u^n_ \lambda (t,x)=\frac{n(n^2-1)\lambda^2}{24}\,.
 	\end{equation}
 	It follows that inequality \eqref{eqn:qp} is not sharp in the limit $t\to\infty$. Therefore, our method can not be directly applied in this case. It is, however believed in physics literature that \eqref{eqn:whitenmm} holds for any real number $n>0$.

\noindent
	\textbf{Acknowledgment:} The author was supported by the NSF  Grant no. 0932078 000, while  he visited the Mathematical Sciences Research Institute in Berkeley, California, during the Fall 2015 semester. The author thanks Jingyu Huang for helpful discussions and encouragement.
	
\bibliographystyle{named}
\bibliography{bib}
\end{document}